\documentclass[12pt]{amsart}
\usepackage[margin=3.0cm]{geometry}

\usepackage{amssymb}
\usepackage{latexsym}
\usepackage{hyperref}
\usepackage{cite}

\usepackage[mathcal]{euscript}
\usepackage[active]{srcltx}

\renewcommand{\phi}{\varphi}

\newcommand{\bbR}{{\mathbb{R}}}

\newcommand{\calX}{{\mathcal{X}}}

\newcommand{\Lie}{{\mathcal{L}}}

\DeclareMathOperator{\id}{id}
\DeclareMathOperator{\im}{im}
\DeclareMathOperator{\pr}{pr}

\DeclareMathOperator{\Span}{Span}
\DeclareMathOperator{\Diff}{Diff}
\DeclareMathOperator{\Graph}{Graph}

\DeclareMathOperator{\Ann}{Ann}

\newcommand*{\longhookrightarrow}{\ensuremath{\lhook\joinrel\relbar\joinrel\rightarrow}}

\newcommand*{\InnDer}{\ensuremath{\relbar\joinrel\stackrel{\mapstochar}{}}}
\newcommand{\Vol}{{\mathrm{vol}}}

\theoremstyle{plain}
\newtheorem{Theorem}{Theorem}[section]
\newtheorem{Proposition}[Theorem]{Proposition}
\newtheorem{Corollary}[Theorem]{Corollary}
\newtheorem{Lemma}[Theorem]{Lemma}

\theoremstyle{definition}
\newtheorem{Definition}[Theorem]{Definition}

\newtheorem{Remark}[Theorem]{Remark}
\newtheorem{Example}[Theorem]{Example}

\newtheorem{Assumption}[Theorem]{Assumption}

\theoremstyle{remark}

\begin{document}

\title[Removable presymplectic singularities]{Removable presymplectic singularities\\ and the local splitting of Dirac structures}

\author[C.~Blohmann]{Christian Blohmann}

\address{Max-Planck-Institut f\"ur Mathematik, Vivatsgasse 7, 53111 Bonn, Germany}
\email{blohmann@mpim-bonn.mpg.de}

\subjclass[2010]{53D18; 53D05, 53D17}


\begin{abstract}
We call a singularity of a presymplectic form $\omega$ \emph{removable in its graph} if its graph extends to a smooth Dirac structure over the singularity. An example for this is the symplectic form of a magnetic monopole. A criterion for the removability of singularities is given in terms of regularizing functions for pure spinors. All removable singularities are poles in the sense that the norm of $\omega$ is not locally bounded. The points at which removable singularities occur are the non-regular points of the Dirac structure for which we prove a general splitting theorem: Locally, every Dirac structure is the gauge transform of the product of a tangent bundle and the graph of a Poisson structure. This implies that in a neighborhood of a removable singularity $\omega$ can be split into a non-singular presymplectic form and a singular presymplectic form which is the partial inverse of a Poisson bivector that vanishes at the singularity. An interesting class of examples is given by log-Dirac structures which generalize log-symplectic structures. The analogous notion of removable singularities of Poisson structures is also studied.
\end{abstract}

\maketitle


\section{Introduction}

Dirac structures were introduced by Courant and Weinstein \cite{CourantWeinstein:1988,Courant:Transactions1990} as generalization and unified description of presymplectic and Poisson structures. The idea is to study presymplectic and Poisson structures on a manifold $M$ in terms of their graphs, i.e., the graphs of the maps $\omega: TM \to T^* M$ and $\Pi: T^*M \to TM$, which are both subbundles of $TM \oplus T^*M$. (The definition of Dirac structures is recalled in Sec.~\ref{sec:DiracRecall}.) Dirac structures have since become a very active field of research and found a wide range of interesting applications. (For recent work see e.g.\ \cite{AlekseevBursztynMeinrenken:2009,Li-BlandMeinrenken:2014,Gualtieri:2014,BrahicFernandes:2014,BursztynCavalcantiGualtieri:2015,AlekseevMeinrenken:2012} and references therein.) In this paper we shall study the following remarkable phenomenon:

Consider the magnetic symplectic form of a magnetic monopole in 2 dimensions, given in Darboux coordinates of $M = T^* \bbR^2$ as
\begin{equation*}
\label{eq:2dMonopole}
\begin{aligned}
 \omega 
 &:= \omega_0 + B \\
 &= dq^i \wedge dp_i + r^{-2} dq^1 \wedge dq^2 \,,
\end{aligned}
\end{equation*}
where $r^2 = (q^1)^2 + (q^2)^2$. This 2-form has a singularity at $M_\mathrm{sing} := T^*_0 \bbR^2$ in the sense that it is only defined on $M \setminus M_\mathrm{sing}$ and cannot be extended to any of the points of $M_\mathrm{sing}$.

The intriguing observation is now that even though $\omega$ cannot be extended to the singularities, its graph extends to a smooth Dirac structure over $M_\mathrm{sing}$ (Section~\ref{sec:GuidingExample}). In this sense the singularity of $\omega$ at $M_\mathrm{sing}$ is removed by the Dirac structure. (For another example see p.~634 of Courant's original paper \cite{Courant:Transactions1990}). If we repeat this with the magnetic symplectic form of a magnetic monopole on $\bbR^3$, however, we find that its singularity is not removable (Example~\ref{ex:DiracMonopole}). This shows that the question whether a singularity is removable or not can be quite subtle. The purpose of this paper is to study the phenomenon of removable singularities in a systematic way.

\subsection{Main results}

The main results of the paper are the following: In Proposition~\ref{prop:SpinorRemove} we observe that a singularity of $\omega$ at $m \in M_\mathrm{sing}$ is removable if and only if there is a regularizing function $f$ on a neighborhood of $m$ such that $f e^{-\omega}$ extends to a smooth differential form at $m$. This implies that $\omega$ cannot be bounded on a neighborhood of $m$ (Cor.~\ref{cor:FiniteNotRemovable}). In this sense all removable singularities are poles. Singularities at which $\omega$ is bounded but not-differentiable are, therefore, not removable. There is an analogous criterion for Poisson structures (Prop.~\ref{prop:PoissonSpinorRemove}).

When all singularities of $\omega$ at $M_\mathrm{sing}$ are removed by a Dirac structure $L$ over $M$, then $M_\mathrm{sing}$ is the set of singular points of $L$, i.e.\ the points at which the dimension of the characteristic distribution $\rho(L)$ is not locally constant. This shows that if we want to understand and classify removable singularities we have to understand the form of Dirac structures at singular points. In Theorem~\ref{thm:SplittingTheorem} we prove a splitting theorem which essentially states that every point of $M$ has a neighborhood isomorphic to $S \times N$ such that the Dirac structure is locally isomorphic to
\begin{equation*}
 L \cong e^{\omega}\bigr(TS \times \Graph(\Pi) \bigr) \,,
\end{equation*}
where $e^{\omega}$ is the gauge transformation or $B$-field transform by a closed 2-form $\omega$ (in the sense of \cite{SeveraWeinstein:2001, Hitchin:2003, Gualtieri:Annals2011}) and $\Pi$ is a Poisson bivector on $N$.

This result enables us to prove a splitting theorem for presymplectic forms with singularities (Theorem~\ref{thm:OmegaSplitting}): A 2-form $\omega$ has a removable singularity at $m \in M_\mathrm{sing}$ if and only if it can be split locally as
\begin{equation*}
 \omega = \omega_{\mathrm{reg}} + \omega_{\mathrm{sing}}
\end{equation*}
into a closed 2-form $\omega_\mathrm{reg}$ that does not have a singularity at $m$ and a closed 2-form $\omega_\mathrm{sing}$, such that (i) $\ker \omega_{\mathrm{sing}}$ is a regular distribution that extends to $m$, and (ii) there is a Poisson bivector $\Pi$ on a neighborhood of $m$ that is a \emph{partial inverse} of $\omega_\mathrm{sing}$, i.e.\ that satisfies
\begin{equation*}
 \omega_{\mathrm{sing}} \Pi \omega_{\mathrm{sing}} = \omega_{\mathrm{sing}} \,,\quad
 \Pi \omega_{\mathrm{sing}} \Pi = \Pi \,.
\end{equation*}
Moreover, the splitting can be chosen such that $\Pi$ vanishes at $m$. In the example of the magnetic monopole in 2 dimensions, $\omega_{\mathrm{reg}} = \omega_0$ and $\omega_{\mathrm{sing}} = B$. This theorem explains why the singularity of a magnetic monopole on $\bbR^3$ is not removable: The kernel of the magnetic field consists, in spherical coordinates, of the radial directions and does not extend to a regular distribution at the origin (Example~\ref{ex:DiracMonopole}).

\subsection{Notation}
The annihilator of a vector subspace $W \subset V$ will be denoted by $W^\circ = \{\alpha \in V^* ~|~ \alpha(W) = 0\}$. For a Dirac structure $L \subset TM \oplus T^*M$ the restricted projections to the tangent and cotangent bundles will be denoted by $\rho := \pr_{TM}|_L$ and $\rho^* := \pr_{T^*M}|_L$. Vectors and vector fields will be denoted by capital Roman letters $X,Y \in TM$ and 1-forms by small Greek letters $\alpha, \beta \in T^*M$, so that typical elements of $TM \oplus T^*M$ will be denoted by $X + \alpha$ and $Y + \beta$. Summation convention: Repeated indices (regardless if upper or lower) will always be summed over.

\subsection{Acknowledgements}
I would like to thank first of all Alan Weinstein for teaching me about Dirac structures and the math department of UC Berkeley, where this project was started, for its hospitality. I would also like to acknowledge Max K\"orfer for his stimulating questions. Earlier versions of Prop.~\ref{prop:SpinorRemove} and Cor.~\ref{cor:Splitting2} are stated in his Bachelor's thesis \cite{Koerfer:Bachelor2014}. I am grateful to Alan Weinstein, Pedro Frejlich, and Madeleine Jotz for their very helpful comments. Finally, I would like to thank the referees whose comments have improved this paper substantially. Def.~\ref{def:logDirac} of log-Dirac structures was suggested by the first referee. Note added in proof: Recently, Bursztyn, Lima, and Meinrenken have proved a powerful splitting theorem \cite{BursztynLimaMeinrenken:2016} which generalizes Theorem~\ref{thm:SplittingTheorem}.

\section{Removable singularities of presymplectic forms}
\label{sec:presymplectic}

\subsection{Singularities of presymplectic structures}

We start by defining the notion of singularity that shall be studied in this paper:

\begin{Definition}
\label{def:Singularity}
Let $M$ be a manifold and $M_\mathrm{sing} \subset M$ a closed subset. A 2-form $\omega \in \Omega^2(M \setminus M_\mathrm{sing})$ is said to have a \emph{singularity} at $m \in M_\mathrm{sing}$ if does not have a smooth extension to a neighborhood of $m$. 
\end{Definition}

Having a smooth extension to a neighborhood $U$ of $m$ means that there is a 2-form $\bar{\omega} \in \Omega^2\bigl( (M \setminus M_\mathrm{sing}) \cup U \bigr)$ such that its restriction to $M \setminus M_\mathrm{sing}$ is $\omega$. An inner point of $M_\mathrm{sing}$ cannot be singular, so, if all points in $M_\mathrm{sing}$ are singular, $M \setminus M_\mathrm{sing}$ is dense in $M$. This is the situation we shall consider:

\begin{Assumption}
\label{assmpt:SingularDense}
For the rest of this paper we assume that $M$ is a smooth manifold of dimension $\geq 2$ and that $M_\mathrm{sing} \subset M$ is a closed subset with dense complement. 
\end{Assumption}

Having a singularity at $m$ means that in local coordinates some of the entries of $\omega_{ij}$ or their derivatives do not converge at $m$. A first rough classification of singularities is given by the following:

\begin{Definition}
Let $\wedge^2 TM$ be equipped with the structure of a normed vector bundle. A singularity of $\omega \in \Omega^2(M \setminus M_\mathrm{sing})$ at $m \in M_\mathrm{sing}$ is said to be \emph{bounded} if there is a neighborhood $U$ of $m$ such that $\| \omega \|$ is bounded on $U \setminus M_\mathrm{sing}$.
\end{Definition}

Note that this definition does not depend on the choice of the norm.

\subsection{Dirac structures}
\label{sec:DiracRecall}

We recall that a \emph{Lie algebroid} is a smooth vector bundle $A \to M$ over a manifold $M$ together with a bundle map $\rho: A \to TM$, called the anchor, and a Lie bracket on the vector space of smooth sections of $A$, such that 
\begin{equation*}
 [a,fb] = \bigl( \rho(a)\cdot f \bigr) b + f[a,b]
\end{equation*}
for all $a, b \in \Gamma^\infty(M,A)$, $f \in C^\infty(M)$.

We recall that a \emph{Dirac structure} on a manifold $M$ consists of a subbundle $L \subset TM \oplus T^*M$ that is maximally isotropic with respect to the symmetric pairing
\begin{equation}
\label{eq:SymPairing}
  \langle X + \alpha, Y + \beta \rangle_+ 
  := \tfrac{1}{2} (i_X \beta + i_Y \alpha)
\end{equation}
and closed under the Dorfman bracket
\begin{equation*}
  [X + \alpha, Y + \beta] = [X,Y] + \Lie_X \beta - i_Y d \alpha \,.
\end{equation*}
It follows that the vector bundle $L \to M$ has the structure of a Lie algebroid given by the Dorfman bracket and the anchor $\rho_L(X + \alpha) = X$. Moreover, the pull-back to $L$ of the antisymmetric pairing defines the \emph{associated Lie algebroid 2-form} $\omega_L \in \Gamma(\wedge^2 L^*)$ given by
\begin{equation*}
 \omega_L(X + \alpha, Y + \beta)
 := - \tfrac{1}{2}( i_X \beta - i_Y \alpha ) \,,
\end{equation*}
which is closed in the Lie algebroid cohomology.

Two obvious candidates for Dirac structures are given by the graphs of a 2-form $\omega \in \Omega^2(M)$ and a bivector field $\Pi \in \calX^2(M) = \Gamma(M, \wedge^2 TM)$,
\begin{equation*}
\begin{aligned}
 \Graph(\omega) &:= \{X + i_X \omega ~|~ X \in TM \} \,, \\
 \Graph(\Pi) &:= \{i_\alpha \Pi + \alpha ~|~ \alpha \in T^*M\} \,.
\end{aligned}
\end{equation*}
One of the advantages of studying 2-form and bivector fields in terms of their graphs is that it gives a unified description of the integrability conditions:

\begin{Proposition}
$\Graph(\omega)$ is a Dirac structure if and only if $\omega$ is closed. $\Graph(\Pi)$ is a Dirac structure if and only if $\Pi$ is Poisson. Moreover, if we identify $\Graph(\omega) = TM$ then the associated Lie algebroid 2-form is $\omega$. If we identify $\Graph(\Pi) = T^*M$ then the associated Lie algebroid 2-form is $\Pi$.
\end{Proposition}

\subsection{The guiding example}
\label{sec:GuidingExample}

As announced in the introduction the following will be our guiding example of a removable singularity:

\begin{Proposition}
The Dirac structure $\Graph(\omega)$ of the magnetic symplectic form $\omega \in \Omega^2( T^*\bbR^2 \setminus T^*_0 \bbR^2 )$ of a magnetic monopole in 2 dimensions defined in Eq.~\eqref{eq:2dMonopole} extends smoothly to a Dirac structure on $T^* \bbR^2$.
\end{Proposition}

Before we give the proof we note that, since $T^*\bbR^2 \setminus T^*_0 \bbR^2$ is dense in  $M := T^* \bbR^2$, the extension of the Dirac structure to the singular locus is unique. In fact, all we have to show is that the topological closure
\begin{equation*}
  L := \overline{\Graph(\omega)}
\end{equation*}
is a smooth vector subbundle of $TM \oplus T^* M$. For this we have to find a basis of local sections of $\Graph(\omega)$ that extends smoothly to a basis over the singular fiber $M_\mathrm{sing} := T_0^*\bbR^2$. The smoothness of the Lie algebroid bracket, of the anchor, and of the Lie algebroid 2-form is automatic since each of these structures is given by the restriction of a smooth map on an ambient manifold (see Remark~\ref{rmk:smoothimplied}).

\begin{proof}
The natural basis of sections of $\Graph(\omega)$ is obtained by inserting the coordinate vector fields of the Darboux coordinates into $\omega$ which yields
\begin{equation*}
  \tilde{a}_i := \frac{\partial}{\partial q^i} + dp_i + B_{ij} dq^j  \,,\quad
  b_i := \frac{\partial}{\partial p_i} - dq^i \,,
\end{equation*}
where $B_{ij} = r^{-2} \varepsilon_{ij}$ are the components of $B = \tfrac{1}{2} B_{ij} dq^i \wedge dq^j$. The sections $\tilde{a}_i$ do not extend to $r=0$ where they have a pole. Instead we can take the following sections:
\begin{equation*}
  a_i 
  := B^{-1}_{ij} \tilde{a}_j 
  =  B^{-1}_{ij} \frac{\partial}{\partial q^j} + dq^i + B^{-1}_{ij} dp_j \,,
\end{equation*}
where
\begin{equation*}
  B^{-1}_{ij} = - r^{2} \varepsilon_{ij} \,.
\end{equation*}
The sections $a_i$ extend smoothly to $r=0$ where they have the value $a_i(0) = dq^i$. We conclude that $\{a_1, a_2, b_1, b_2\}$ is a set of smooth sections of $\Graph(\omega)$ that extend smoothly to $r=0$ where they remain linearly independent.
\end{proof}

Let us compute the rest of the Lie algebroid structure. The anchor is given by
\begin{equation*}
  \rho(a_i) = B^{-1}_{ij} \frac{\partial}{\partial q^j} \,,\quad
  \rho(b_i) =  \frac{\partial}{\partial p_i} \,.
\end{equation*}
The characteristic distribution on $M = T^* \bbR^2$ given by the image of the anchor has two leaves: the manifold $T^*\bbR^2 \setminus T^*_0 \bbR^2$ on which $\omega$ is defined and the singular fiber $T_0^*\bbR^2$. The Lie bracket of the sections of $L$ is given by
\begin{equation*}
 [a_i, a_j] 
 = \frac{\partial B^{-1}_{ij} }{\partial q^k} a_k 
 = - 2\varepsilon_{ij} q^k a_k \,,\quad
 [a_i, b_j] = 0 = [b_i, b_j] \,.
\end{equation*}
The 2-form $\omega_L$ of the Dirac structure $L$ is given by
\begin{equation*}
 \omega_L(a_i, a_j) = B^{-1}_{ij} \,,\quad 
 \omega_L(a_i, b_j) = - B^{-1}_{ij} \,,\quad 
 \omega_L(b_i, b_j) = 0 \,, 
\end{equation*}
which vanishes on the singular fiber. Finally, let us verify that $L$ is not the graph of a closed 2-form or a Poisson bivector. The following criterion is obvious yet useful:

\begin{Proposition}
\label{prop:AnchorSurjective}
Let $L$ be a Dirac structure over $M$ and $m \in M$ a point.
\begin{itemize}
 \item[(i)] $L$ is the graph of a 2-form in a neighborhood of $m$ if and only if the anchor is surjective at $m$.
 \item[(ii)] $L$ is the graph of a bivector field in a neighborhood of $m$ if and only if $L \cap T_m M = 0$.
\end{itemize}
\end{Proposition}

\begin{proof}
(i) If $\rho$ is surjective at $m$ then it is surjective in a neighborhood $U$ of $m$. For reasons of dimension it is a local isomorphism of vector bundles. The map
\begin{equation*}
  TU \stackrel{\rho^{-1}}{\longrightarrow} L \stackrel{\mathrm{pr}}{\longrightarrow} T^*U
\end{equation*}
can be identified with a bilinear form on $TU$ which is antisymmetric because $L$ is isotropic.

(ii) Since $\pr_{T^*M}(L) = (L \cap TM)^\circ$, the projection $\pr_{T^*M}$ is surjective at $m$ if and only if $L \cap T_m M = 0$. The rest of the proof is dual to (i).
\end{proof}

It is easy to see that for the Dirac structure of the example neither the anchor nor the projection to $T^* \bbR^2$ is surjective at $r=0$. Therefore, it is neither the graph of a 2-form nor the graph of a bivector field. 

\begin{Remark}
The example immediately generalizes to magnetic symplectic forms $\omega = \omega_0 + B$ on $M = T^*Q \setminus T^*_q Q$, where $B$ is the pull-back to $M$ of a symplectic form on $Q \setminus \{q\}$ that is the inverse of a Poisson bivector with a zero at $q$. We will come back to this case in Sec.~\ref{sec:logsymplectic}.
\end{Remark}

\subsection{Removable singularities}

\begin{Definition}
A singularity of a presymplectic form $\omega \in \Omega^2(M \setminus M_\mathrm{sing})$ at $m \in M_\mathrm{sing}$ is called \emph{removable in its graph} if the closure of $\Graph(\omega)$ in $TM \oplus T^*M$ is a Dirac structure over a neighborhood of $m$.
\end{Definition}

As no other notion of removable singularity shall be considered in this paper we will write ``removable'' for ``removable in its graph''.

Since by the general Assumption~\ref{assmpt:SingularDense} the set $M \setminus M_\mathrm{sing}$ is dense in $M$, the Dirac structure that removes a singularity is unique. Moreover, if a singularity at $m$ is removable then all singularities in a neighborhood of $m$ are removable. The analogous notion of removable singularities of Poisson bivector fields will be studied in Sec.~\ref{sec:Poisson}

\begin{Remark}
\label{rmk:smoothimplied}
Because $M \setminus M_\mathrm{sing}$ is dense, it need only be checked whether the closure of the graph is a smooth vector subbundle. All other properties follow: Since the symmetric pairing on $TM \oplus T^*M$ is a smooth map, the closure of the graph is isotropic. Since the projection $TM \oplus T^*M \to TM$ is a smooth map, the anchor extends smoothly to the closure. Since the Lie bracket on sections of the graph is the restriction of the Courant bracket, which is smoothly defined for all sections of $TM \oplus T^*M$, it extends smoothly to the closure. Since all structure maps are smooth the Jacobi identity and the Leibniz rule of the Lie algebroid bracket hold on the closure, as well. Moreover, since the antisymmetric pairing on $TM \oplus T^*M$ is smooth, the 2-form on the graph extends to a 2-form on its closure which is closed in the Lie algebroid cohomology.
\end{Remark}

Assume that all singularities in $M_\mathrm{sing}$ can be removed by the Dirac structure $L \to M$. It follows from Prop.~\ref{prop:AnchorSurjective} that the anchor of $L$ is not surjective at $M_\mathrm{sing}$, since otherwise it would be the graph of a presymplectic form which contradicts the definition of singularity. Therefore, $M_\mathrm{sing}$ consists of the union of all leaves of non-zero codimension of the characteristic distribution $\rho(L)$. Since, furthermore, $M_\mathrm{sing}$ has no inner points, it is the set of singular points of the Lie algebroid, i.e.\ the set of points at which the rank of the anchor is not locally constant. This suggests that the study of removable singularities is closely related to the study of Dirac structures at singular points. We come back to this observation in Sec.~\ref{sec:SplittingTheorem}

\subsection{Pure spinors and regularizing functions}

In order to find useful criteria for a singularity to be removable we recall the relation between pure spinors and Dirac structures (see \cite{Gualtieri:Annals2011}).

Consider the sum $T \oplus T^*$ of a vector space $T$ and its dual, which is equipped with the non-degenerate symmetric bilinear form \eqref{eq:SymPairing}. The Clifford algebra $\mathrm{Cl}(T \oplus T^*)$ is defined as the free algebra generated by $T \oplus T^*$ divided by the relations $v^2 = \langle v, v \rangle_+$ for all $v \in T \oplus T^*$.

The Clifford algebra has a faithful representation on the space $\wedge^\bullet T^*$ of forms on $T$ defined for the generators $v = (X+\alpha) \in T \oplus T^*$ by
\begin{equation*}
 (X + \alpha) \cdot \phi := i_X \phi + \alpha \wedge \phi \,.
\end{equation*}
The annihilator $\Ann \phi := \{v\in T\oplus T^* ~|~ v \cdot \phi = 0 \}$ of a given form is an isotropic subspace. Conversely, every isotropic subspace is the annihilator of a form which is unique up to a nonzero scalar. A form $\phi$ is called a \emph{pure spinor} \cite{Chevalley:Spinors1954} if $\Ann(\phi)$ is maximally isotropic, that is, of dimension $n$.

A smooth family of maximally isotropic subspaces of $TM \oplus T^*M$ is then given by a smooth family of pure spinors $\phi \in \Omega(M)$. The integrability condition of a Dirac structure can be deduced from the relation
\begin{equation*}
 [X + \alpha, Y + \beta] \cdot \phi = (X + \alpha)(Y + \beta) \cdot d\phi \,,
\end{equation*}
which holds for all $X + \alpha, Y + \beta \in \Ann(\phi)$. The condition $d\phi = 0$ is sufficient for $\Ann(\phi)$ to be a Dirac structure but not necessary.

The graph of a presymplectic form $\omega$ is the annihilator of the pure spinor $\phi = e^{-\omega} = 1 - \omega + \frac{1}{2} \omega \wedge \omega + \ldots$, since
\begin{equation*}
 0 = (X + \alpha) \cdot e^{-\omega}
 = (- i_X \omega + \alpha ) \wedge e^{-\omega}
\end{equation*}
holds if and only if $\alpha = i_X \omega$.

\begin{Proposition}
\label{prop:SpinorRemove}
A singularity of the closed 2-form $\omega \in \Omega^2(M \setminus M_\mathrm{sing})$ at $m \in M_\mathrm{sing}$ is removable if and only if there is a smooth function $f \in C^\infty(U)$ on a neighborhood $U$ of $m$ such that $f e^{-\omega}$ extends to a nowhere vanishing differential form on $U$.
\end{Proposition}
\begin{proof}
Assume that the singularity is removable. Then there is a pure spinor field $\phi$ on a neighborhood $U$ of $m$ such that $\Ann(\phi)$ is the smooth Dirac structure extending $\Graph(\omega) = \Ann(e^{-\omega})$. Since the annihilators of $\phi$ and $e^{-\omega}$ are equal on $U \setminus M_\mathrm{sing}$ there is a function $\tilde{f}$ on $U \setminus M_\mathrm{sing}$ such that
\begin{equation*}
 \phi\bigr|_{U \setminus M_\mathrm{sing}} = \tilde{f} e^{-\omega} = \tilde{f} - \tilde{f}\omega 
      + \tilde{f} \tfrac{1}{2}\omega \wedge \omega + \ldots \,.
\end{equation*}
Since this pure spinor extends smoothly to the singular points in $U$, every term on the right hand side must extend smoothly. Therefore, $\tilde{f}$ must extend to a smooth function $f$ on $U$.

Conversely, assume that $f \in C^\infty(U)$ is such that $f e^{-\omega}$ extends to a smooth, nowhere vanishing differential form $\phi$ on $U$. Then $\Ann(\phi)$ is a smooth subbundle of $TM \oplus T^* M$ that restricts to $\Graph(\omega)$. According to Remark~\ref{rmk:smoothimplied} this is all we need to show.
\end{proof}

We call the function $f$ of the proposition a \emph{regularizing} function of the singularity. From the last theorem we can derive three rather crude yet useful obstructions to the removability of a singularity:

\begin{Corollary}
\label{cor:FiniteNotRemovable}
Let $\omega \in \Omega^2(M \setminus M_\mathrm{sing})$ be a closed 2-form. Let
$f := \| e^{-\omega}\|^{-1}$ which is a smooth function on $M \setminus M_\mathrm{sing}$. A singularity of $\omega$ at $m \in M_\mathrm{sing}$ is not removable if any one of the following obstructions applies:
\begin{itemize}
 \item[(i)] $f:= \| e^{-\omega}\|^{-1}$ does not extend continuously to $m$.
 \item[(ii)] $f$ extends continuously to $m$ but $f(m) \neq 0$.
 \item[(iii)] The singularity is bounded.
\end{itemize}
\end{Corollary}
\begin{proof}
Assume that the singularity is removable. By Prop.~\ref{prop:SpinorRemove} the Dirac structure is the annihilator of a pure spinor $\phi$. Without loss of generality we can choose $\phi$ to be normalized, such that on $M \setminus M_\mathrm{sing}$ we have $1 = \| \phi \| = \| g e^{-\omega} \| = |g| \, f^{-1}$, where $g$ is the regularizing function.

(i) The regularizing function $g$ is continuous at $m$, so $|g|$ is a continuous extension of $f$ to $m$.

(ii) Let $\bar{f}$ be the continuous extension of $f$ to $m$. Assume that $\bar{f}(m) = |g(m)| \neq 0$. Then there is a neighborhood of $m$ on which $g^{-1}$ is defined. Hence, $\phi g^{-1} = e^{-\omega} = 1 - \omega + \ldots$ extends smoothly to a neighborhood of $m$. This implies that $\omega$ extends smoothly to that neighborhood which contradicts the assumption that $\omega$ as a singularity at $m$.

(iii) Assume that the singularity is bounded. Then there is a constant $C$ such that $\| \omega \| \leq C$ on a neighborhood of $m$. We can assume that the norm is submultiplicative, i.e.\  $\| \omega \wedge \omega' \| \leq \| \omega \| \, \| \omega' \|$. Then the regularizing function satisfies
\begin{equation*}
 |g| = \frac{1}{\| e^{-\omega} \|} \geq \frac{1}{e^{\|\omega\|}} \geq \frac{1}{e^{C}} \,,
\end{equation*}
so it does not vanish at $m$. Hence, obstruction (ii) applies.
\end{proof}

The following example shows that even if none of the three obstructions of Cor.~\ref{cor:FiniteNotRemovable} applies a singularity can still fail to be removable.

\begin{Example}
\label{ex:rhoCylinder}
Let $M := \bbR^{3}$, $M_\mathrm{sing} := \{(x,y,z) ~|~ x = y = 0 \}$ and $\rho := \sqrt{x^2 + y^2}$. Consider the closed 2-form
\begin{equation*}
  \omega
  = - d(\rho^{-1}) \wedge dz
  = \frac{x\,dx + y\,dy}{\rho^3} \wedge dz \,,
\end{equation*}
which has a singularity at $\rho = 0$ that is not bounded. Equipping the bundle $\wedge^\bullet T^*\bbR^3$ with the norm induced by the metric for which the monomials of the coordinate 1-forms are an orthonormal basis we obtain
\begin{equation*}
 f = \|e^{-\omega}\|^{-1} = \frac{\rho^2}{\sqrt{1 + \rho^4}} \,,
\end{equation*}
which extends smoothly to $\rho = 0$ where it vanishes. We conclude that none of three obstructions of Cor~\ref{cor:FiniteNotRemovable} applies. If the singularity is removable then $f$ must be the absolute value of a regularizing function $g$. Since $M \setminus f^{-1}(0)$ is connected, the only smooth functions $g$ satisfying $f = |g|$ are $g = \pm f$. However, the spinor
\begin{equation*}
 \phi := f e^{-\omega} = f - \frac{x\,dx + y\,dy}{\rho \sqrt{1 + \rho^4} } \wedge dz
\end{equation*}
does not extend smoothly to $\rho=0$. Hence, by Prop.~\ref{prop:SpinorRemove} the singularity is not removable.
\end{Example}

The upshot of the pure spinor approach is the following procedure to determine whether a singularity at $m$ is removable or not:
\begin{enumerate}
 \item Choose local coordinates $\{x^1, \ldots, x^n\}$ on a neighborhood $U \ni m$.
 \item Equip $\wedge^\bullet T^*M$ with a convenient  fiber-wise metric, e.g.\ by letting the standard basis of ordered monomials of the coordinate 1-forms $dx^i$ be orthonormal.
 \item Compute the function $f = \| e^{-\omega} \|^{-1}$. Check whether the obstructions of Cor.~\ref{cor:FiniteNotRemovable} apply.
 \item Find all smooth functions $g$ on $U$ such that $|g|$ is a continuous extension of $f$. If $U \setminus f^{-1}(0)$ is connected, the only such functions are $g = \pm f$.
 \item Check whether the $2k$-form $g\omega^k$ extends smoothly to a neighborhood of $m$ for every $0 \leq 2k \leq n$.
\end{enumerate}

\begin{Corollary}
\label{cor:SplitRemove}
Assume that we can split the closed 2-form $\omega \in \Omega^2(M \setminus M_\mathrm{sing})$ on a neighborhood of a singular point $m \in M_\mathrm{sing}$ as
\begin{equation*}
 \omega = \omega_\mathrm{reg} + \omega_\mathrm{sing} \,,
\end{equation*}
where $\omega_\mathrm{reg}$ extends smoothly to a neighborhood of $m$. Then the singularity of $\omega$ is removable if and only if the singularity of $\omega_\mathrm{sing}$ is.
\end{Corollary}
\begin{proof}
Let $\bar{\omega}_\mathrm{reg}$ be the smooth extension of $\omega_\mathrm{reg}$. Let $f$ be a smooth function. The form $\phi = f e^{-\omega}$ extends smoothly to a non-vanishing form on a neighborhood of $m$ if and only if $\phi' := \phi \wedge e^{\bar{\omega}_\mathrm{reg}} = f e^{-\omega} \wedge e^{\bar{\omega}_\mathrm{reg}} = f e^{-\omega_\mathrm{sing}}$ does.
\end{proof}

\begin{Remark}
It is quite obvious that a Dirac structure $L$ has a smooth extension to a singular point if and only if any gauge transformed Dirac structure $e^{\omega_\mathrm{reg}} L$ has a smooth extension as well, which implies Cor.~\ref{cor:SplitRemove} without using spinors. The spinorial point of view shows that $f$ is a regularizing function of $\omega$ if and only if it is a regularizing function of $\omega_\mathrm{sing}$.
\end{Remark}

The symplectic form of the guiding example, $\omega = \omega_0 + B$, is of this type with $\omega_\mathrm{reg} = \omega_0$ the standard symplectic form of the cotangent bundle and $\omega_\mathrm{sing} = B$ the 2-form of a magnetic monopole in 2 dimensions.

\section{The splitting theorem for Dirac structures}
\label{sec:SplittingTheorem}

\subsection{Motivation and statement of the theorem}

In order to gain further insight into the structure of removable singularities we will prove a splitting theorem for Dirac structures. The classic example of a local splitting theorem is the Weinstein splitting theorem for Poisson structures which states that in a neighborhood of a point $m \in M$ a Poisson manifold is the product of a symplectic manifold and a manifold with a Poisson structure that vanishes at $m$ (Theorem 2.1 in \cite{Weinstein:Splitting1983}). Our theorem will be similar.

We recall that the product of two Dirac manifolds  $(L,M)$ and $(L',M')$ is given by the product of the vector bundles
\begin{equation*}
\begin{split}
 L \times L'
 &:= \{ (X, X') \oplus (\alpha, \alpha')~|~ 
  X + \alpha \in L\,, X' + \alpha' \in L'\} \\
 &\subset T(M\times M') \oplus T^*(M \times M') \,.
\end{split}
\end{equation*}
The anchor is given by $\rho_{L\times L'}(a,a') = \rho(a)+\rho(a') \in T(M \times M')$, and the Lie bracket by $[(a,a'),(b,b')] = ([a,b], [a',b'])$ for sections $a,b \in \Gamma(L)$, $a',b' \in \Gamma(L')$ which extends to more general sections by the Leibniz rule.

The splitting theorem will also involve the deformation of a Dirac structure $L$ in the direction of a closed 2-form $\omega$, called gauge transformation in~\cite{SeveraWeinstein:2001} and $B$-field transform in \cite{Gualtieri:Annals2011}:
\begin{equation}
\label{eq:BfieldTransform}
 e^{\omega} L := \{ X + \alpha + i_X \omega ~|~ X + \alpha \in L \} \,,
\end{equation}
which is again a Dirac structure. In fact, it was shown in \cite{Gualtieri:Annals2011} that the automorphism group of the exact Courant algebroid $TM \oplus T^*M$ is given by the semi-direct product $\Omega_\mathrm{cl}^2(M) \rtimes \Diff(M)$, where the additive group $\Omega_\mathrm{cl}^2(M)$ acts by Eq.~\eqref{eq:BfieldTransform} and the diffeomorphisms by push-forward.

By Prop.~\ref{prop:AnchorSurjective} a Dirac structure that removes a singularity at $m$ is neither the graph of a presymplectic form nor of a bivector field if $\rho(L)$ does not have full dimension and $L \cap TM$ is not zero at $m$. This suggests that $\rho(L)$ and $L \cap TM$ determine the directions in $M$ that are relevant for a local splitting theorem of Dirac structures.

\begin{Remark}
\label{rmk:NullNonIntegrable}
Since $L \cap TM = \pr_{T^*M}(L)^\circ$ and since $\dim \bigl( \pr_{T^* M}(L) \bigr)$ is lower semi-continuous, as it is the case for every image of a map of vector bundles, $\dim(L \cap TM)$ is upper semi-continuous. The dimension of a smooth singular foliation in the sense of Stefan-Sussman is lower semi-continuous. Therefore, $L \cap TM$ is integrable if and only if it is regular.
\end{Remark}

\begin{Theorem}[Splitting theorem]
\label{thm:SplittingTheorem}
Let $L \subset TM \oplus T^*M$ be a Dirac structure on $M$. Every $m \in M$ has an embedded neighborhood $\phi: S \times N \hookrightarrow M$, $(s,n) \mapsto m$ such that
\begin{equation*}
 \phi^* L = e^{\omega}\bigl( TS \times \Graph(\Pi) \bigr) \,,
\end{equation*}
where $\Pi$ is a Poisson bivector on $N$ and $\omega$ is a closed 2-form that vanishes at $(s,n)$. \end{Theorem}

\begin{Remark}
\label{rmk:SplittingGeo}
Before we prove this theorem, let us comment on the geometry of the splitting. The two embedded submanifolds
\begin{equation*}
 \tilde{S} := \phi(S \times \{n\}) \,,\qquad
 \tilde{N} := \phi(\{s\} \times N) 
\end{equation*}
intersect transversely in $m$. Since the image of the anchor is invariant under gauge transformations, the submanifold $\tilde{S}$ is tangent to the presymplectic leaf of $L$ through $m$. Moreover, since $\omega$ vanishes at $(s,n)$, we have
\begin{equation*}
 T_m \tilde{S} = L \cap T_m M \,,
\end{equation*}
which is the obstruction for $L$ to be the graph of Poisson bivector. Therefore, $\tilde{N}$ is a \emph{maximal Poisson transversal}, i.e.\ a maximal Poisson submanifold that intersects the presymplectic leaf of $L$ through $m$ transversely and symplectically.
\end{Remark}

\subsection{Two lemmas}

For the proof of Theorem~\ref{thm:SplittingTheorem} we need two basic lemmas:

\begin{Lemma}
\label{lem:ClosedOmega}
Let $X + \alpha \in \Gamma^\infty(M,L)$ be a section of the Dirac structure $L$. If $X$ does not vanish at $m$, then there is a closed 2-form $\omega$ defined on a neighborhood of $m$ such that locally $\alpha = i_X \omega$. Moreover, if $\alpha_m = 0$ then $\omega$ can be chosen such that it vanishes at $m$.
\end{Lemma}

\begin{proof}
Since $X$ does not vanish at $m$, there are local coordinates $x, y^1, \ldots, y^{n-1}$ such that $m=0$ and $X = \frac{\partial}{\partial x}$. Since by the isotropy of the Dirac structure $i_X \alpha = \langle X + \alpha, X +\alpha \rangle = 0$, we have $\alpha = \alpha_i dy^i$. Let $f_i$ be functions such that $\frac{\partial f_i}{\partial x} = \alpha_i$ and $f_i(x=0) = 0$. Set $\gamma := f_i dy^i$, the de Rham differential of which is given by
\begin{equation*}
 \omega := d\gamma
 = \alpha_i dx \wedge dy^i + \frac{\partial f_i}{\partial y^j} \,  dy^j \wedge dy^i
  \,.
\end{equation*}
We conclude that $i_X \omega = i_{\frac{\partial}{\partial x}} \omega = \alpha_i dy^i = \alpha$. Moreover, since $f_i(x=0) = 0$ the partial derivatives in the direction of the hyperplane given by $x=0$ vanish as well, $\frac{\partial f_i}{\partial y^j}(x = 0) =0$. It follows that $\omega(0) = \alpha_i(0)\, dx \wedge dy^i$, which vanishes when $\alpha(0) = 0$.
\end{proof}

\begin{Lemma}
\label{lem:LineSplit}
Let $L$ be a Dirac structure $L$ over $M$. Assume that in a neighborhood of $m \in M$ there is a section of the form $X + 0 = X$ such that the vector field $X$ does not vanish at $m$. Let $M' \subset M$ be a submanifold of codimension 1 through $m$ such that $X_m \notin T M'$. Then $L$ is locally the product of the Dirac structure $\Span\{X\}$ restricted to the integral curve of $X$ and a Dirac structure over $M'$.
\end{Lemma}

\begin{proof}
Choose local coordinates $\{x, y^1, \ldots, y^{n-1}\}$ such that $X = \frac{\partial}{\partial x}$, that $m=0$, and that $M'$ is locally isomorphic to the hyperplane given by $x=0$. We can find a basis of local sections of the form
\begin{equation*}
 a = \frac{\partial}{\partial x} + 0 \,,\quad
 b_i = A_{ij} \frac{\partial}{\partial y^j} + B_{ij} dy^j \,,
\end{equation*}
where the indices $i$, $j$ run from $1$ through $n-1$. Since $[\rho(a),\rho(b_i)]$ does not have a component in the direction of $\frac{\partial}{\partial x}$ we have $[a,b_i] = F_{ij}b_j$ for some functions $F_{ij}$. Let $U_{ij}$ be the solution of the differential equation $\frac{\partial U_{ij}}{\partial x} = - U_{ik} F_{kj}$ with initial value $U_{ij}(x=0) = \delta_{ij}$. For $x$ small enough $U_{ij}$ is an invertible matrix, so that $\{a, b'_i := U_{ij} b_j \}$ is again a basis of local sections of $L$. We compute
\begin{equation*}
\begin{split}
 [a, b'_i] 
 &= (\rho(a) U_{ij}) b_j + U_{ik} F_{kj} b_j \\
 &= \Bigl( \frac{\partial U_{ij}}{\partial x} + U_{ik} F_{kj} \Bigr) b_j \\
 &= 0 \,.
\end{split}
\end{equation*}
We further note that $[b'_i, b'_j] = C_{ij}^k b'_k$ for some functions $C_{ij}^k$. Since $[a, [b'_i, b'_j]] = [[a,b'_i], b'_j] + [b'_i, [a, b'_j]] = 0$ it follows that
\begin{equation*}
 0 = [a, [b'_i, b'_j]] = [a, C_{ij}^k b'_k] = \frac{\partial C_{ij}^k}{\partial x} b'_k \,,
\end{equation*}
which shows that $C_{ij}^k$ does not depend on $x$.

It follows that $L$ is locally isomorphic to the Lie algebroid spanned by $\{a\}$ restricted to the line given by $y^1 = \ldots = y^{n-1}$ and the Lie algebroid $\{ b'_1, \ldots, b'_{n-1} \}$ restricted to the normal hyperplane given by $x=0$.
\end{proof}

\subsection{Proof of Theorem~\ref{thm:SplittingTheorem}}

The proof is by induction over
\begin{equation*}
 p = \dim (L \cap T_m M) \,.
\end{equation*}
If $p=0$, then $\pr_{T^*M}: L \to M$ is surjective at $m$ and, thus, surjective on a neighborhood of $m$. Therefore, $L$ restricted to that neighborhood is already the graph of a Poisson structure $\Pi$. So the statement is trivially true for $S$ a point and $\omega = 0$.

Let now $p > 0$. Since $L \cap T_m M \subset \rho(L_m)$, there is a nowhere vanishing vector field $X$ defined on a neighborhood $U$ of $m$ such that $X$ is tangent to $\rho(L)$ and $X_m \in L \cap T_m M$. This means that there is a local section $X + \alpha$ of $L$ and that $\alpha_m = 0$. We can now apply Lemma~\ref{lem:ClosedOmega}: There is a closed 2-form $\tilde{\omega}$ that vanishes at $m$ such that $i_X \tilde{\omega} = \alpha$.

It follows that $X$ is a section of the gauge transformed Dirac structure $e^{-\tilde{\omega}} L|_U$. Let $C$ be the integral curve of $X$ through $m$ and let $M'$ be a codimension 1 transversal to $C$ through $m$. Now we can apply Lemma~\ref{lem:LineSplit}, which states that, after suitably shrinking $C$ and $U$, there is an embedding $\psi: C \times M' \hookrightarrow M$ with $\psi(m,m) = m$ such that
\begin{equation*}
 \psi^* e^{-\tilde{\omega}} L = TC \times L' \,,
\end{equation*}
where $L'$ is a Lie algebroid over $M'$. This is equivalent to
\begin{equation*}
 \psi^* L = e^{\psi^*\tilde{\omega}} (TC \times L') \,,
\end{equation*}
where we have used the relation $\psi^* e^{-\tilde{\omega}} =  e^{-\psi^* \tilde{\omega}} \psi^*$ for gauge transformations, which was proved in \cite{Gualtieri:Annals2011}.

In order to apply the induction hypothesis to $L'$, we need to show that $\dim (L' \cap T_m M') = \dim (L \cap T_m M) - 1 = p -1$. Recall that $\rho^*: L \to T^*M$ is the map induced by the projection $TM \oplus T^*M \to T^*M$. Since $\tilde{\omega}$ vanishes at $m$, we have $L_m =  e^{-\tilde{\omega}} L_m$. We thus obtain
\begin{equation*}
\begin{split}
 \dim( \rho^* L_m )
 &= \dim ( \rho^* e^{-\tilde{\omega}} L_m )
 = \dim ( \rho^* \psi^* e^{-\tilde{\omega}} L_m ) \\
 &= \dim\bigl( \rho^* (T_m C \times L'_m ) \bigr)
 = \dim(\rho^* T_m C) + \dim(\rho^* L'_m ) \\
 &= \dim(\rho^* L'_m) \,.
\end{split}
\end{equation*}
Using the fundamental relation $(L \cap TM)^\circ = \rho^* L$ for Dirac structures, which implies $\dim M - \dim (L \cap TM) = \dim \rho^* L$, we compute
\begin{equation*}
\begin{split}
 \dim (L' \cap T_m M') 
 &= \dim \bigl( (\rho^* L'_m)^\circ \bigr) \\
 &= \dim M' - \dim ( \rho^* L'_m ) \\
 &= \dim M  - 1 - \dim ( \rho^* L_m ) \\
 &= \dim (L \cap T_m M) - 1 \,.
\end{split}
\end{equation*}

Now we can use the induction hypothesis which states that there is an embedding $\chi: S' \times N \hookrightarrow M'$ with $\chi(s',n) = m$ such that
\begin{equation*}
 \chi^* L' = e^{\omega'} \bigl( TS' \times \Graph(\Pi) \bigr) \,,
\end{equation*}
where $\omega'$ is a closed 2-form on $S' \times N$ that vanishes at $(s,n)$ and where $\Pi$ is a Poisson bivector on $N$.

Composing the two embeddings, we obtain the embedding
\begin{equation*}
 \phi := \psi \circ (\id_C \times \chi):
 C \times S' \times N \longhookrightarrow  M \,,
\end{equation*}
which maps $(m,s',n)$ to $m$. Setting $S := C \times S'$ and $s := (m,s')$, we can view this as an embedding $\phi: S \times N \hookrightarrow M$ with $\phi(s,n) = m$. We obtain the following equality of Lie algebroids
\begin{equation*}
\begin{aligned}
 \phi^* L 
 &= (\id_C \times \chi)^* \psi^* L \\
 &= (\id_C \times \chi)^* e^{\psi^*\tilde{\omega}} 
    (TC \times L') \\
 &= e^{(\id_C \times \chi)^* \psi^* \tilde{\omega}} 
    (\id_C \times \chi)^* (TC \times L') \\
 &= e^{\phi^*\tilde{\omega}} 
    \bigl( TC \times e^{\omega'}(TS' \times \Graph(\Pi) )\bigr) \\
 &= e^{\omega} \bigl( TS \times \Graph(\Pi) \bigr) \,, 
\end{aligned}
\end{equation*}
where $\omega := \phi^*\tilde{\omega} + \pi^* \omega'$ and where $\pi:C \times S' \times N \to S' \times N$ is the projection.

Finally, we observe that since $\tilde{\omega}$ and $\omega'$ are both closed, $\omega$ is closed as well. Furthermore, since $\tilde{\omega}$ is zero at $m = \phi(s,n)$, $\phi^* \tilde{\omega}$ is zero at $(s,n)$. Since $\omega'$ is zero at $(s',n) = \pi(s,n)$, $\pi^* \omega'$ is zero at $(s,n)$. It follows that $\omega$ is zero at $(s,n)$.
\qed

\subsection{Corollaries of the splitting theorem}

\begin{Corollary}
\label{cor:omegavanish}
A Dirac structure $L$ is locally isomorphic to a product of the form
\begin{equation*}
 L|_U \cong TS \times \Graph(\Pi) \,,
\end{equation*}
where $U \cong S \times N$ is a neighborhood of $m$ and $\Pi$ a Poisson bivector on $N$, if and only if the distribution $L \cap TM$ is regular at $m$.
\end{Corollary}

\begin{proof}
Assume that $L \cap TM$ is regular at $m$. Since $L \cap TM$ is an involutive distribution, it follows that it is integrable. Therefore, we can find local coordinates $(x^1, \ldots, x^p, y^1, \ldots, y^q)$ such that $a_i := \frac{\partial}{\partial x^i}$ are sections of $L$. This means that in the proof of Thm.~\ref{thm:SplittingTheorem} we can take $\omega=0$. Conversely, if $\omega=0$ then $\phi^*(L \cap TM) = TS$, which is regular.
\end{proof}

\begin{Corollary}
\label{cor:ThreeWaySplitting}
Let $L \subset TM \oplus T^*M$ be a Dirac structure on $M$. Every point $m \in M$ has an embedded neighborhood $\phi: S \times N_1 \times N_2 \hookrightarrow M$, $(s, n_1, n_2) \mapsto m$, such that
\begin{equation}
\label{eq:SplittingTheorem}
 \phi^* L = e^{\omega}\bigl( TS \times \Graph(\omega_1) \times \Graph(\Pi_2) \bigr) \,,
\end{equation}
where $\omega$ is a closed 2-form that vanishes at $(s, n_1, n_2)$, $\omega_1$ is a symplectic form on $N_1$, and $\Pi_2$ is a Poisson bivector on $N_2$ that vanishes at $n_2$. Moreover, $\Pi_2 = 0$ if and only if $\rho(L)$ is regular at $m$.
\end{Corollary}

\begin{proof}
After applying the splitting theorem~\ref{thm:SplittingTheorem}, we apply the Weinstein splitting theorem \cite{Weinstein:Splitting1983} to the Poisson manifold $(N, \Pi)$: After shrinking $N$, there is an isomorphism $N_1 \times N_2 \to N$, $(n_1, n_2) \mapsto n$ of Poisson manifolds, where $(N_1, \omega_1)$ is a symplectic manifold and $(N_2, \Pi_2)$ a Poisson manifold such that $\Pi_2$ vanishes at $n_2$.

Moreover, since $\rho\bigl( \Graph(\omega_1) \bigr) = TN_1$, we have $TS \times TN_1 \subset \rho(\phi^* L)$. Since $\Pi_2$ vanishes at $n_2$, we also have $T_m S \oplus T_m N_1 = \rho_m (\phi^* L)$. Assume that $\rho(L)$ is regular at $m$. Then the dimension of $\rho(\phi^* L)$ is locally constant, which implies that $TS \times TN_1 = \rho(\phi^* L)$. It follows that $T^*N_2 = \rho(\phi^* L)^\circ = \phi^*(L \cap T^* M)$. Therefore, $\Graph(\Pi_2) = T^* N_2$, so that $\Pi_2 = 0$. Conversely, if $\Pi_2 = 0$ it follows that $\rho(\phi^* L) = \phi^*(L \cap T^*M) = T^* N_2$, which is of constant dimension, so that the dimension of $\rho(L)$ is locally constant, as well.
\end{proof}

\begin{Corollary}
If both $\rho(L)$ and $L \cap TM$ are regular distributions at $m$, the splitting~\eqref{eq:SplittingTheorem} takes the form
\begin{equation*}
 \phi^* L = TS \times \Graph(\omega_1) \times T^* N_2 \,,
\end{equation*}
which is Proposition 4.1.2 of Courant's thesis \cite{Courant:Transactions1990}.
\end{Corollary}

\begin{proof}
Follows immediately from Cor.~\ref{cor:omegavanish} and Cor.~\ref{cor:ThreeWaySplitting}.
\end{proof}

\begin{Corollary}
\label{cor:Splitting2}
Let $L \subset TM \oplus T^*M$ be a Dirac structure on $M$. Every $m \in M$ has an embedded neighborhood $\phi: S' \times N' \hookrightarrow M$, $(s,n) \mapsto m$ such that
\begin{equation*}
 \phi^* L = e^{\omega'}\bigl( TS' \times \Graph(\Pi') \bigr) \,,
\end{equation*}
where $\Pi'$ is a Poisson bivector on $N'$ that vanishes at $n$ and $\omega'$ is a closed 2-form.
\end{Corollary}

\begin{proof}
Since $\Graph(\omega_1) = e^{\omega_1}(TN_1)$ we can rewrite Eq.~\eqref{eq:SplittingTheorem} as
\begin{equation*}
 \phi^* L = e^{\omega + \pi^* \omega_1} 
  \bigl( TS \times \Graph(\omega_1) \times \Graph(\Pi_2) \bigr) \,,
\end{equation*}
where $\pi: S \times N_1 \times N_2 \to N_1$ denotes the projection. Setting $\omega' := \omega + \pi^*\omega_1$, $S' := S \times N_1$, and $N' := N_2$ yields the statement.
\end{proof}

\begin{Remark}
Cor.~\ref{cor:Splitting2} is the equivalent of Theorem~1.4 in \cite{AbouzaidBoyarchenko:2006} for generalized complex structures.
\end{Remark}

\begin{Remark}
\label{rmk:SpinorSplit}
Once we have locally split the Dirac structure as $e^{\omega}\bigl(TS \times \Graph(\Pi) \bigr)$, we can also factor the corresponding pure spinor $\psi \in \Omega(S\times N)$ as
\begin{equation*}
 \psi = e^{-\omega} \wedge \bigl( e^{-\Pi} \cdot \Vol_N \bigr) \,,
\end{equation*}
where $\Vol_N$ is a volume form on $N$ and where we have used the form of the spinor of a Poisson structure given by Prop.~\ref{prop:Pispinor}. 
\end{Remark}

The Poisson bivectors of Theorem~\ref{thm:SplittingTheorem} and Corollary~\ref{cor:Splitting2} can be viewed, respectively, as the maximal and minimal Poisson part of the Dirac structure at $m$. Let us spell out the splittings in local coordinates.

Let $\{x^1, \ldots, x^p\}$ be local coordinates of $S$ and $\{y^1, \ldots, y^q\}$ local coordinates of $N$ such that $m=0$. The closed 2-form $\omega$ and the Poisson bivector $\Pi$ are of the form
\begin{equation}
\label{eq:Defomega}
\begin{aligned}
 \omega &= 
  \tfrac{1}{2} \omega^{xx}_{ij} dx^i \wedge dx^j 
  + \omega^{xy}_{i\alpha} dx^i \wedge dy^\alpha 
  + \tfrac{1}{2} \omega^{yy}_{\alpha\beta} dy^\alpha \wedge dy^\beta   
  \\
 \Pi &= 
  \tfrac{1}{2} \Pi_{\alpha\beta}(y)\, \frac{\partial}{\partial y^\alpha} \wedge 
    \frac{\partial}{\partial y^\beta} \,,
\end{aligned}
\end{equation}
where the Latin indices run through $1 \leq i,j \leq p$ and the Greek indices through $1 \leq \alpha,\beta \leq q$. 

Thm.~\ref{thm:SplittingTheorem} and Cor.~\ref{cor:Splitting2} both state that there is a local frame of $L$ of the form
\begin{equation}
\label{eq:StandardFrame}
\begin{aligned}
 a_i &=  
  \frac{\partial}{\partial x^i} + (\omega^{xx}_{ij} dx^j + \omega^{xy}_{i\beta} dy^\beta ) \\
 b_\alpha &= 
  \Pi_{\alpha\beta} 
  \frac{\partial}{\partial y^\beta } +
  ( dy^\alpha - \Pi_{\alpha\beta} \omega^{xy}_{j \beta} dx^j 
  + \Pi_{\alpha\beta} \omega^{yy}_{\beta\gamma} dy^\gamma ) \,.
\end{aligned}
\end{equation}
In terms of this standard basis the Lie bracket is given by
\begin{equation*}
 [a_i,a_j] = 0 \,,\quad
 [a_i, b_\alpha] = 0 \,,\quad
 [b_\alpha, b_\beta] = \frac{\partial\Pi_{\alpha\beta}}{\partial y^\gamma} b_\gamma \,.
\end{equation*}
The anchor is given by
\begin{equation*}
 \rho(a_i) = \frac{\partial}{\partial x^i} \,,\quad
 \rho(b_\alpha) = \Pi_{\alpha\beta} \frac{\partial}{\partial y^\beta} \,.
\end{equation*}
The Lie algebroid 2-form is given by
\begin{equation*}
\begin{gathered}
 \omega_L(a_i,a_j) = \omega^{xx}_{ij} \,,\quad
 \omega_L(a_i,b_\alpha) = \omega^{xy}_{i\beta} \Pi_{\alpha\beta} \,,\\
 \omega_L(b_\alpha,b_\beta) = \Pi_{\alpha\beta} 
 + \Pi_{\alpha\alpha'}\omega^{yy}_{\alpha'\beta'} \Pi_{\beta'\beta}\,.
\end{gathered}
\end{equation*}
Finally, we observe that in the situation of Thm.~\ref{thm:SplittingTheorem} $\omega$ vanishes at $0$, whereas in the situation of Cor.~\ref{cor:Splitting2} $\Pi$ vanishes for $y=0$.

\section{Splitting of presymplectic forms with singularities}
\label{sec:PresymplecticSplitting}

\subsection{Presymplectic structures with partial inverses}

Let $\omega \in \Omega^2(M \setminus M_\mathrm{sing})$ be a symplectic form and $\Pi = \omega^{-1}$ its inverse Poisson bivector. If $\Pi$ extends to a smooth bivector $\Bar{\Pi}$ on a neighborhood of a singular point $m \in M_\mathrm{sing}$ of $\omega$, then the closure of $\Graph(\omega)$ in that neighborhood is given by $\Graph(\Bar{\Pi})$. Hence, the singularity of $\omega$ is removable. This case can be generalized to presymplectic forms.

\begin{Definition}
Let $\omega \in \Omega^2(M)$ be viewed as map $TM \to T^*M$ and $\Pi \in \calX^2(M)$ viewed as map $T^*M \to TM$. We say that $\omega$ and $\Pi$ are \emph{partial inverses} to each other if
\begin{equation}
\label{eq:PartialInverses}
 \omega \Pi \omega = \omega \,,\quad
 \Pi \omega \Pi = \Pi \,.
\end{equation}
\end{Definition}

\begin{Proposition}
Let $\omega \in \Omega^2(M)$ be a closed 2-form. The following are equivalent:
\begin{itemize}
 \item[(i)] $\omega$ has a partial inverse Poisson bivector on a neighborhood of $m \in M$.
 \item[(ii)] $\omega$ is regular (i.e.\ of locally constant rank) at $m \in M$.
\end{itemize}
\end{Proposition}
\begin{proof}
Assume that $\Pi$ is a partial inverse to $\omega$ on some neighborhood $U$ of $m$ which we can assume, without loss of generality, to be connected. Note that the rank of every smooth map of vector bundles is lower semi-continuous. In particular, the function $\dim \im \omega$ is lower-semi-continuous and, hence, $\dim \ker \omega$ is upper semi-continuous. It follows from Eqs.~\eqref{eq:PartialInverses} that $\id_{TM} -  \Pi\omega$ is the projection to the kernel of $\omega$, which is a smooth map of vector bundles. It follows that $\im \ker \omega$ is also lower semi-continuous. We see that the dimension of the kernel of $\omega$ must be both upper and lower semi-continuous, hence, locally constant. Since $U$ is connected the dimension must be constant.

Assume now that $\omega$ is regular at $m \in M$. Since $\omega$ is closed $\ker \omega$ is an integrable distribution on some coordinate neighborhood $U \ni m$. Let $\{x^1, \ldots, x^p, y^1, \ldots, y^q \}$ be coordinates such that $\{ \frac{\partial}{\partial x^1} , \ldots,  \frac{\partial}{\partial x^p} \}$ spans $\ker\omega$. Then $\omega = \tfrac{1}{2} \omega_{ij} dy^i \wedge dy^j$ where the matrix $\omega_{ij}$ is non-degenerate. The bivector field $\Pi := \frac{1}{2} (\omega^{-1})^{ij} \frac{\partial}{\partial y^i} \wedge \frac{\partial}{\partial y^j}$ is a Poisson bivector that is a partial inverse of $\omega$.
\end{proof}

\begin{Remark}
The matrix function $\omega_{ij}$ in the proof does not depend on the $x$-coordinates since $0 = d( \frac{\partial}{\partial x^k} \InnDer \omega) = \Lie_{\frac{\partial}{\partial x^k}} \omega = \frac{1}{2}\frac{\partial \omega_{ij}}{\partial x^k} dy^i \wedge dy^k$.
\end{Remark}

The graph of $\omega$ can be expressed in terms of a partial inverse Poisson bivector $\Pi$ as follows: Since $TM = \ker \omega \oplus \im \Pi$, we have
\begin{equation*}
 \Graph \omega 
 = \Graph\omega|_{\ker \omega} \oplus \Graph \omega|_{\im \Pi} \,.
\end{equation*}
The map $\Pi\omega$ is the projection onto the image of $\Pi$, so that
\begin{equation*}
\begin{split}
 \Graph \omega |_{\im \Pi}
 &= \bigl\{ \Pi\omega(X) + \omega\bigl(\Pi\omega(X) \bigr) ~|~  X\in TM \bigr\} \\
 &= \bigl\{ \Pi\bigl(\omega(X)\bigr) + \omega(X) ~|~  X\in TM \bigr\} \\ 
 &= \Graph \Pi|_{\im\omega} \\
 &= \Graph \Pi|_{(\ker\omega)^\circ} 
 \,.
\end{split}
\end{equation*}
We conclude that
\begin{equation}
\label{eq:omegaPiGraph}
 \Graph \omega = \ker \omega \oplus \Graph \Pi|_{(\ker\omega)^\circ} \,,
\end{equation}
where we view $\ker \omega$ as subspace of $TM \oplus T^*M$. This suggests the following:

\begin{Proposition}
\label{prop:PartialRemovable}
Let the closed 2-form $\omega \in \Omega^2(M \setminus M_\mathrm{sing})$ have a singularity at $m \in M_\mathrm{sing}$. If there is a neighborhood $U$ of $m$ such that
\begin{itemize}
 \item[(i)] $\ker \omega$ extends to a smooth regular distribution on $U$,
 \item[(ii)] there is a Poisson bivector $\Pi$ on $U$ which is a partial inverse of $\omega$ on $U \setminus M_\mathrm{sing}$,
\end{itemize}
then the singularity is removable.
\end{Proposition}
\begin{proof}
Let $\overline{\ker \omega}$ be the smooth regular extension (i) of the distribution $\ker\omega$ and $\Pi$ the bivector (ii).  We will show that, in analogy to Eq.~\eqref{eq:omegaPiGraph}, the closure of $\Graph\omega$ is the direct sum of $\overline{\ker \omega}$ and of the graph of $\Pi$ restricted to the annihilator of $\overline{\ker \omega}$. 

Assumption (i) means that there is a smooth local frame $\{X_1, \ldots, X_p\} \subset \Gamma(U, \overline{\ker\omega})$. Since $\im\omega$ is the annihilator of $\ker\omega$, there is also a smooth frame $\{\beta_1, \ldots \beta_q\} \subset \Gamma(U,\Ann\bigl(\overline{\ker\omega}) \bigr) = \Gamma(U, \overline{\im\omega})$. And since $\{\beta_1, \ldots, \beta_q\}$ spans the image of $\omega$, there is a set $\{Z_1, \ldots, Z_q\}$ of smooth vector fields on $U$, such that $\beta_i = \omega(Z_i)$ on $U \setminus M_\mathrm{sing}$. We define
\begin{equation*}
 \eta_i := \Pi(\beta_i) + \beta_i
\end{equation*}
for all $i \in \{1, \ldots, q\}$, which is a linearly independent family of smooth sections of $TU \oplus T^*U$. By Eq.~\eqref{eq:PartialInverses} we have $\omega(Z_i) = \omega\bigl( \Pi(\omega(Z_i)) \bigr) = \omega(Y_i)$ where $Y_i := \Pi(\omega(Z_i))$. It follows that
\begin{equation*}
 \eta_i = \Pi(\beta_i) + \beta_i = \Pi\bigl(\omega(Y_i) \bigr) + \omega(Y_i)
 = Y_i + \omega(Y_i) \,,
\end{equation*}
which shows that all $\eta_i$ lie in the closure of $\Graph\omega$. The vector fields $X_i$ lie in the closure of $\Graph\omega$, as well, because $X_i$ is in the kernel of $\omega$. Since the rank of the graph is $\dim M = p+q$, the linearly independent set of smooth sections $\{X_1, \ldots X_p, \eta_1, \dots, \eta_q\}$ spans the closure of $\Graph\omega$ over $U$ which is, therefore, a smooth vector subbundle of $TU \oplus T^* U$.
\end{proof}

The following examples show that neither condition (i) nor condition (ii) of the proposition can be dropped.

\begin{Example}
This is a variant of Example~\ref{ex:rhoCylinder}. Let $M := \bbR^{3}$, $M_\mathrm{sing} := \{(x,y,z) ~|~ x = y = 0 \}$ and $\rho := \sqrt{x^2 + y^2}$. Consider the closed 2-form
\begin{equation*}
  \omega
  = d(\log\rho) \wedge dz
  = \frac{x\,dx + y\,dy}{\rho^2} \wedge dz \,,
\end{equation*}
which has singularities at $\rho = 0$. The Poisson bivector
\begin{equation*}
 \Pi =  - \Bigl( x \frac{\partial}{\partial x} + y \frac{\partial}{\partial y} \Bigr)
 \wedge \frac{\partial}{\partial z} \,,
\end{equation*}
is a partial inverse of $\omega$ such that $\Pi\omega$ is the projection to the $(\rho,z)$-planes in cylindrical coordinates. Since $\Pi$ extends smoothly to $\rho = 0$, condition (ii) is satisfied. The kernel of $\omega$ is spanned by the vector field $y \frac{\partial}{\partial x} - x \frac{\partial}{\partial y}$. The integrating foliation of concentric circles around the $z$-axis does not extend to a regular foliation at $\rho = 0$, so condition (i) is not met.

The candidate for the absolute value of the regularizing function is given by
\begin{equation*}
 f = \|e^{-\omega}\|^{-1} = \frac{\rho}{\sqrt{1 + \rho^2}} \,.
\end{equation*}
This function does not extend smoothly to $\rho = 0$. Moreover, since the open set $M \setminus f^{-1}(0)$ is connected, the only functions $g$ such that $f = |g|$ are $g = \pm f$. We conclude that the singularities at $\rho = 0$ are not removable.
\end{Example}

\begin{Example}
Let $M = \bbR^2$, $M_\mathrm{sing} = \{ 0 \}$, and $\omega = \sqrt{x^2 + y^2} \, dx \wedge dy$ which does not extend smoothly to 0. The kernel of $\omega$ extends to the zero distribution on all of $\bbR^2$, so condition (i) is satisfied. However, the inverse $\Pi = (x^2 + y^2)^{-\frac{1}{2}} \, \frac{\partial}{\partial x} \wedge \frac{\partial}{\partial y}$ does not extend smoothly to $0$, so (ii) is not satisfied. Since the singularity is bounded, it is not removable by Cor.~\ref{cor:FiniteNotRemovable} (iii).
\end{Example}

The last example can be easily generalized: Every presymplectic form $\omega$ on $M \setminus M_\mathrm{sing}$ that has an extension which is continuous but not smooth at $m \in M_\mathrm{sing}$ has a singularity at $m$. Since $\omega$ is bounded, this singularity is not removable.

\begin{Example}[Dirac monopole]
\label{ex:DiracMonopole}
Let $M = \bbR^3$, $M_\mathrm{sing} = \{ 0 \}$, and $\omega$ be given in spherical coordinates $\{r, \theta, \phi\}$ by $\omega = d\theta \wedge \sin \theta d\phi$, i.e.\ by radially constant area forms of spheres centered at the origin. The kernel of $\omega$ is spanned by the radial vector field $\frac{\partial}{\partial r}$. This is a distribution that does not extend to a smooth regular distribution at the origin, so condition (i) is not satisfied.

In cartesian coordinates $\{x^1, x^2, x^3\}$ $\omega$ takes the form
\begin{equation*}
 \omega = \varepsilon_{ijk} \frac{x^i}{r^3} \, dx^j \wedge dx^k \,,
\end{equation*}
where $\varepsilon_{ijk}$ denotes the totally antisymmetric tensor with $\varepsilon_{123} = 1$.
The presymplectic form $\omega$ has a partial inverse Poisson bivector given by $\Pi = \frac{1}{\sin\theta} \frac{\partial}{\partial \phi} \wedge \frac{\partial}{\partial \theta}$. In cartesian coordinates we have
\begin{equation*}
 \Pi = - \varepsilon_{ijk} r x^i \,
 \frac{\partial}{\partial x^j}\wedge \frac{\partial}{\partial x^k} \,,
\end{equation*}
from which we see that $\Pi$ has a $C^1$-differentiable extension to $r=0$, so the $C^1$-differentiable version of condition (ii) is satisfied.

For the candidate of a regularizing function we get
\begin{equation*}
 f := \| e^{-\omega} \|^{-1} = \frac{r^2}{\sqrt{1+r^4}} \,.
\end{equation*}
The function extends smoothly to $r=0$ and $f(0) = 0$ so the obstructions of Cor.~\ref{cor:FiniteNotRemovable} do not apply. Since $M \setminus f^{-1}(0)$ is connected, $f$ must be a regularizing function if the singularity is to be removable. However, the spinor field
\begin{equation*}
  \phi 
  = f - f\omega 
  = f +
  \varepsilon_{ijk} \frac{x^i}{r\sqrt{1+r^4}} \, dx^j \wedge dx^k 
\end{equation*}
is not continuous at $r=0$. Hence, by Prop~\ref{prop:SpinorRemove}, the singularity at $r=0$ is not removable.
\end{Example}


\subsection{The splitting theorem for presymplectic forms}

\begin{Theorem}
\label{thm:OmegaSplitting}
The closed 2-form $\omega \in \Omega^2(M \setminus M_\mathrm{sing})$ has a removable singularity at $m \in M_\mathrm{sing}$ if and only if $\omega$ can be locally split as
\begin{equation*}
 \omega = \omega_{\mathrm{reg}} + \omega_{\mathrm{sing}} \,,
\end{equation*}
into a closed 2-form $\omega_\mathrm{reg}$ that extends smoothly to a neighborhood of $m$ and a closed 2-form $\omega_\mathrm{sing}$ that satisfies conditions (i) and (ii) of Prop.~\ref{prop:PartialRemovable}, such that the partial inverse Poisson tensor of $\omega_{\mathrm{sing}}$ vanishes on the presymplectic leaf through $m$ of the extended Dirac structure.
\end{Theorem}

\begin{proof}
Assume that there is such a splitting $\omega = \omega_{\mathrm{reg}} + \omega_{\mathrm{sing}}$. It follows from Cor.~\ref{cor:SplitRemove} and Prop.~\ref{prop:PartialRemovable} that the singularity at $m$ is removable. 

Conversely, let the singularity at $m$ be removable and $L$ the Dirac structure that extends the graph of $\omega$ to a neighborhood of $m$. The 2-form $\omega_\mathrm{reg}$ is the closed 2-form that exists by Cor.~\ref{cor:Splitting2} such that $\phi^* L = e^{\omega_\mathrm{reg}}\bigl( TS \times \Graph(\Pi) \bigr)$. Since $L$ is the graph of a presymplectic form $\omega$, its anchor and, hence, the anchor of $e^{-\omega_\mathrm{reg}} \phi^*L$ is surjective on $M \setminus M_\mathrm{sing}$. It follows that outside of the singular set $M_\mathrm{sing}$ the Poisson bivector  $\Pi$ has a partial inverse given in local coordinates \eqref{eq:Defomega} by
\begin{equation*}
 \omega_\mathrm{sing} 
 :=  \tfrac{1}{2} \Pi^{-1}_{\alpha\beta}(y)\, d y^\alpha \wedge d y^\beta \,.
\end{equation*}
Since $\omega_\mathrm{sing}$ is closed because it is the difference of two closed 2-forms. Moreover, the kernel of $\omega_{\mathrm{sing}}$ extends to the foliation given by the fibers of the projection $S \times N \to N$. Cor.~\ref{cor:Splitting2} also states that $\Pi$ vanishes for $y=0$, as noted at the end of Sec.~\ref{sec:SplittingTheorem}.
\end{proof}

\begin{Remark}
The splitting of Theorem~\ref{thm:OmegaSplitting} is optimal in the sense that $\omega_\mathrm{sing}$ has the smallest possible rank.
\end{Remark}

\subsection{Log-symplectic and log-Dirac structures}
\label{sec:logsymplectic}

A \emph{log-sym\-plec\-tic} (also called $b$-symplectic or $b$-log-symplectic) structure on a manifold $M^{2n}$ is a Poisson bivector field $\Pi$ such that $\Pi^n$ has only non-degenerate zeros, i.e.\ viewed as section of the canonical bundle $\wedge^{2n} TM$ it intersects the zero section transversely \cite{Radko:2002,GuilleminMirandaPires:2014}. This implies that the zero locus $M_\mathrm{sing}$ is a submanifold of codimension 1, that $\Pi$ has an inverse symplectic form $\omega_\mathrm{log}$ on $M \setminus M_\mathrm{sing}$, and that $\omega$ has a Darboux type standard form in a neighborhood of a singular point when using a logarithmic coordinate in the normal direction (whence the name). Since $\omega_\mathrm{log}$ is the inverse of a Poisson vector that continues smoothly to $M_\mathrm{sing}$, its singularities are removable. The corresponding Dirac structure can be identified with the log-tangent bundle given by $\Pi^n$. In terms of the spinor of a Dirac structure log-symplectic structures can be characterized as follows:

\begin{Proposition}
\label{prop:LogSymplecticSpinor}
A Dirac structure $L$ on $M^{2n}$ is the graph of a log-symplectic structure if every point of $M$ has a neighborhood $U$, such that the pure spinor $\phi \in \Omega(U)$ defining $L|_U$ has the following two properties: 
\begin{itemize}
 \item[(i)] The function $\phi_0$ has only non-degenerate zeros.
 \item[(ii)] The section of the canonical bundle $\phi_{2n}$ is nowhere vanishing.
\end{itemize}
\end{Proposition}
\begin{proof}
Condition (ii) requires $\phi_{2n}$ to be a local volume form. This is the case iff for every nowhere vanishing vector field $X \in \calX(U)$ the inner derivative $i_X \phi_{2n}$ is nowhere vanishing. It follows that (ii) holds iff the only vector field in the annihilator of $\phi$ is $X = 0$, i.e.~iff $L \cap TU = 0$. By Prop.~\ref{prop:AnchorSurjective} this is the case iff $L$ is the graph of a bivector field.

Assume that Condition (ii) is satisfied. Then by Prop.~\ref{prop:Pispinor} every local spinor is of the form $\phi = e^{-\Pi} \cdot \Vol_U$ for a volume form $\Vol_U$ and a Poisson bivector $\Pi$ on an open subset $U \subset M$. It follows that $\phi_0 = (-\Pi)^n \cdot \Vol_U$, which has only non-degenerate zeros if and only if $\Pi^n$ has only non-degenerate zeros. 
\end{proof}

Prop.~\ref{prop:LogSymplecticSpinor} suggests to generalize the notion of log-symplectic structures to Dirac structures by dropping the requirement that the Dirac structure is the graph of a Poisson bivector:

\begin{Definition}
\label{def:logDirac}
A Dirac structure is called a \emph{log-Dirac structure} if it satisfies condition (i) of Prop.~\ref{prop:LogSymplecticSpinor}.
\end{Definition}

Generalizations of complex log-symplectic structures in the framework of generalized complex geometry have been studied in \cite{CavalcantiGualtieri:2015}. Here are a few straight-forward observations about log-Dirac structures:

\begin{Proposition}\mbox{ }
\label{prop:logDirac1}

\begin{itemize}

 \item[(i)] Let $\Pi$ be a Poisson bivector on $M$. Then $\Graph(\Pi)$ is a log-Dirac structure if and only if $\Pi$ is a log-symplectic structure.

 \item[(ii)] The gauge transformation $e^\omega L$ of a Dirac structure $L$ is a log-Dirac structure if and only if $L$ is a log-Dirac structure.
 
 \item[(iii)] The product $\Graph \omega \times L'$ of the graph of a presymplectic structure on $M$ and a Dirac structure on $M'$ is a log-Dirac structure if and only if $L'$ is a log-Dirac structure. 

\end{itemize}
\end{Proposition}
\begin{proof}
(i) Follows immediately from Prop.~\ref{prop:LogSymplecticSpinor} since the Dirac structure $\Graph(\Pi)$ is given by the spinor $\phi = e^{-\Pi} \cdot \Vol_M$

(ii) If $L$ is given by the spinor $\phi$, then $e^\omega L$ is given by the spinor $\phi' = e^{-\omega}\phi$. Hence $\phi_0 = \phi'_0$.

(iii) Let $\phi = e^{-\omega} \in \Omega(M)$ be the spinor of $\Graph \omega$ and $\phi' \in \Omega(M')$ the spinor of $L'$. Then the degree zero part of the spinor $\psi \in \Omega(M \times M')$ of $\Graph \omega \times L'$ is given by $\psi_0(m,m') = \phi_0(m) \phi'_0(m') = \phi'_0(m')$. The function $\psi_0$ has non-degenerate zeros if $0$ is a regular value, which is the case if and only if $0$ is a regular value of $\phi'_0$.
\end{proof}

\begin{Proposition}
A Dirac structure $L$ over $M$ is a log-Dirac structure if and only if every point $m \in M$ has an embedded neighborhood $\psi: S \times N \hookrightarrow M$, $(s, n) \mapsto m$ such that
\begin{equation*}
 \psi^* L = e^\omega(TS \times L') \,,
\end{equation*}
where $L'$ is a log-symplectic structure on $N$ and $\omega$ is a closed 2-form. 
\end{Proposition}
\begin{proof}
First, note that the Dirac structure $TS$ is given by the spinor $\phi = 1$, which shows that $TS$ is trivially a log-Dirac structure. Prop.~\ref{prop:logDirac1} (ii) and (iii) then imply that $e^\omega(TS \times L')$ is a log-Dirac structure if and only if $L'$ is a log-Dirac structure. Therefore, if $L' = \Graph(\Pi)$ for $\Pi$ log-symplectic, then $L$ is a log-Dirac structure.

Conversely, assume that $L$ is a log-Dirac structure. By Thm.~\ref{thm:SplittingTheorem} we have over some neighborhood of every point $m \in M$ a splitting with $L' = \Graph(\Pi)$ for a Poisson bivector $\Pi$ on $N$. Prop.~\ref{prop:logDirac1} (i) then states that $\Pi$ must be log-symplectic.
\end{proof}

\begin{Example}
Let $\omega_\mathrm{log}$ be a log-symplectic structure on $M$ with Poisson bivector $\Pi_\mathrm{log}$. The cotangent bundle $\pi: T^* M \to M$ is equipped with the singular symplectic structure $\omega = \omega_0 + \pi^* \omega_\mathrm{log}$, which is defined on $T^*(M \setminus M_\mathrm{sing})$ and has singularities on $T^*M|_{M_\mathrm{sing}}$. While $\omega$ fails to be log-symplectic, Cor.~\ref{cor:SplitRemove} tells us that the singularities of $\omega$ are removable. After trivializing the cotangent bundle on some open neighborhood $N \subset M$ as $T^*N \cong S \times N$, the pure spinor of the Dirac structure takes the form $\phi = e^{\omega_0} \wedge \bigl( e^{-\Pi_\mathrm{log}} \cdot \Vol_N \bigr)$. This shows that the cotangent bundle of a log-symplectic manifold is a log-Dirac manifold. By the same reasoning we see that the cotangent bundle of every log-Dirac manifold is again a log-Dirac manifold.
\end{Example}

\section{Removable singularities of Poisson structures}
\label{sec:Poisson}

\subsection{Removable singularities}

The definition of removable singularities of Poisson structures is completely analogous to the presymplectic case:

\begin{Definition}
A Poisson bivector $\Pi \in \calX^2(M \setminus M_\mathrm{sing})$ is said to have a singularity at $m \in M_\mathrm{sing}$ if $\Pi$ does not have a smooth extension to a neighborhood of $m$. A singularity at $m$ is called removable if the closure of $\Graph(\Pi)$ in $TM \oplus T^*M$ is a Dirac structure over a neighborhood of $m$.
\end{Definition}

More aspects of singularities of Poisson structures are analogous to the presymplectic case: Since the set $M\setminus M_\mathrm{sing}$ is dense, the Lie algebroid extending $\Graph(\Pi)$ is unique. By the same argument as in Remark~\ref{rmk:smoothimplied} it only needs to be checked if the closure of $\Graph(\Pi)$ in $TM \oplus T^*M$ is a smooth vector bundle. The smoothness of the rest of the Dirac structure is automatic.

Assume that all singularities of $\Pi$ are removable so that $L := \overline{\Graph(\Pi)}$ is a smooth Lie algebroid over $M$. Since $\pr_{T^*M}(L) = (L \cap TM)^\circ$, the singular points are those at which the null-distribution $L \cap TM$ has positive dimension. Since the set of points $m$ where $L \cap T_m M =0$ is $M \setminus M_\mathrm{sing}$ and thus dense, $L \cap TM$ is non-regular for all points in $M_\mathrm{sing}$. It follows from Remark~\ref{rmk:NullNonIntegrable} that there are no points in $M_\mathrm{sing}$ where the null-distribution $L \cap TM$ is locally integrable. This is in stark contrast to the presymplectic case.

\subsection{The pure spinor approach}

For the pure spinor approach to the singularities of a Poisson bivector we view multivector fields and differential forms as elements of the Clifford algebra $\mathrm{Cl}(TM \oplus T^*M)$ in the natural way. Using the notation $XY \equiv X\wedge Y$ for vector fields $X$, $Y$ and $\alpha\beta \equiv \alpha \wedge \beta$ for 1-forms $\alpha$, $\beta$, the commutation relations in the Clifford algebra are $XY = - YX$, $\alpha\beta = - \beta\alpha$, and $X \alpha + \alpha X = \langle \alpha, X \rangle$. We recall that there is a natural action of $\mathrm{Cl}(TM \oplus T^*M)$ on a differential form $\phi \in \Omega^\bullet(M)$ defined by $\alpha \cdot \phi = \alpha \wedge \phi$ and $X \cdot \phi = i_X \phi$. The following observation is well-known (e.g.\ \cite{EvensLu:1999} or p.~87 of \cite{Gualtieri:Annals2011}):

\begin{Proposition}
\label{prop:Pispinor}
Let $\Pi$ be a bivector on an orientable manifold $M$. Then $\Graph(\Pi)$ is the annihilator of the pure spinor field
\begin{equation*}
 \phi := e^{-\Pi} \cdot \Vol_M \,,
\end{equation*}
where $\Vol_M$ is a volume form on $M$.
\end{Proposition}
\begin{proof}
Follows from a short calculation in the Clifford algebra.
\end{proof}

If $M$ is not orientable, this spinor still exists locally which suffices to obtain a criterion for the removability of singularities of Poisson structures that is analogous to the presymplectic case.

\begin{Proposition}
\label{prop:PoissonSpinorRemove}
A singularity of the Poisson bivector $\Pi \in \Omega^2(M \setminus M_\mathrm{sing})$ at $m \in M_\mathrm{sing}$ is removable if and only if there is a smooth function $f \in C^\infty(U)$ on a neighborhood $U$ of $m$ such that $f e^{-\Pi}$ extends to a smooth, nowhere vanishing multivector field on all of $U$.
\end{Proposition}
\begin{proof}
Let $\Vol_U$ be a volume form on a neighborhood $U$ of $m$. By the last proposition $\Graph(\Pi)$ is the annihilator of $\phi = e^{-\Pi} \cdot \Vol_U$. By the same argument as in the proof of Prop.~\ref{prop:SpinorRemove} the singularity is removable if and only if there is a smooth function $f$ such that $f\phi = fe^{-\Pi} \cdot \Vol_U$ extends smoothly to a nowhere vanishing differential form on all of $U$. This is the case if and only if $fe^{-\Pi}$ extends smoothly to a nowhere vanishing multivector field on $U$.
\end{proof}

\begin{Remark}
As corollary of this proposition we obtain the analogue of the obstructions to the removability of a singularity given in Cor.~\ref{cor:FiniteNotRemovable}, replacing $\omega$ by $\Pi$ throughout.
\end{Remark}

\subsection{Remark on the splitting of singular Poisson bivectors}

There is no good analog of Theorem~\ref{thm:OmegaSplitting} for a Poisson structure $\Pi$ for the following reasons: First, unlike for presymplectic forms, the sum of two Poisson bivectors is generally not Poisson. Second, there is no Poisson bivector analog of the gauge transformation ($B$-field transform). That is, while $e^\Pi L := \{ X + \Pi(\alpha) + \alpha ~|~ X + \alpha \in L\}$ is a maximally isotropic subbundle, it is in general not closed under the Dorfman bracket. All we can say is that the Poisson bivector of Thm.~\ref{thm:SplittingTheorem} can be viewed as the maximal regular part $\Pi_\mathrm{reg}$ of a Poisson bivector in the neighborhood of a removable singularity. In fact, since the 2-form $\omega$ of Thm.~\ref{thm:SplittingTheorem} vanishes at the singularity, $\omega$ does not have a partial Poisson inverse on any submanifold containing the singularity. Thm.~\ref{thm:SplittingTheorem} also implies that the original Poisson bivector $\Pi$ can be reconstructed from $\Pi_\mathrm{reg}$ and $\omega$. However, the expression obtained after a calculation using the standard frame \eqref{eq:StandardFrame} is rather complicated and does not lead to further geometric insight.

\appendix

\section{The local form of Dufour and Wade}

In Theorem 3.2 of \cite{DufourWade:2008} it was shown that there is a basis of local sections of a Dirac structure $L$ in a neighborhood of a point $m$ of the following form:
\begin{equation}
\label{eq:DufourWadeBasis}
\begin{aligned}
 \tilde{a}_i &= 
  \Bigl( \frac{\partial}{\partial x^i}
  + A_{i\beta} \frac{\partial}{\partial y^\beta} \Bigr) + B_{ij} dx^j \\
 \tilde{b}_\alpha &= 
  \Pi_{\alpha\beta} \frac{\partial}{\partial y^\beta} +
  \bigl( dy^\alpha - A_{i\alpha} dx^i \bigr) \,.
\end{aligned}
\end{equation}
It follows from the defining properties of a Dirac structure that the functions $A_{i\beta}$ and $\Pi_{\alpha\beta}$ must vanish for $y=0$, that $B_{ij} = - B_{ij}$, and that $\Pi := \frac{1}{2}\Pi_{\alpha\beta} \frac{\partial}{\partial y^\beta} \wedge \frac{\partial}{\partial y^\beta}$ is a Poisson bivector.

The vector fields $\rho(\tilde{a}_i)$ of Eq.~\eqref{eq:DufourWadeBasis} can be viewed as affine connection
\begin{equation*}
 \nabla_{\frac{\partial}{\partial x^i}} 
 = \frac{\partial}{\partial x^i}
  + A_{i\beta} \frac{\partial}{\partial y^\beta} 
\end{equation*}
on the normal bundle $NS \cong \bbR^p \times \bbR^q \to \bbR^p \cong S$ of the characteristic leaf through $m = 0$. Furthermore, it was shown in  \cite{DufourWade:2008} that $[\rho(a_i), \Pi] = 0$, which means that the transverse Poisson bivector $\Pi$ is invariant under the connection. This implies that parallel transport induces isomorphisms of the Poisson manifolds of the fibers, so that a Dirac structure in a tubular neighborhood of a characteristic leaf $S$ can be described by the same geometric data as in \cite{Vorobjev:2000}.

However, this result, which is obtained by basic linear algebra, does not yield a local splitting theorem for which it is necessary to show that a \emph{flat} connection compatible with $\Pi$ can be found. The existence of such a flat connection is implied by Theorem~\ref{thm:SplittingTheorem}, as we can simply take $\nabla_{\frac{\partial}{\partial x^i}} = \frac{\partial}{\partial x^i}$ in the local coordinates of Eq.~\eqref{eq:StandardFrame}.

From Theorem~\ref{thm:SplittingTheorem} we obtain a basis in the form \eqref{eq:DufourWadeBasis} as linear combination of the standard frame~\eqref{eq:StandardFrame} (with $\omega^{yy} = 0$ as in Cor.~\ref{cor:Splitting2}) given as
\begin{equation*}
 \tilde{a}_i := a_i - \omega^{xy}_{i\alpha} b_{\alpha} \,,\quad
 \tilde{b}_\alpha := b_\alpha \,.
\end{equation*}
The matrices $A$ and $B$ of \eqref{eq:DufourWadeBasis} are given by
\begin{equation*}
 A_{i\beta} = - \omega^{xy}_{i\alpha} \Pi_{\alpha\beta} \,,\quad
 B_{ij} = \omega^{xx}_{ij} 
   + \omega^{xy}_{i\alpha} \Pi_{\alpha\beta} \omega^{xy}_{j \beta} \,.
\end{equation*}

It is not clear how to obtain from \eqref{eq:DufourWadeBasis} a basis of the form \eqref{eq:StandardFrame}. Let us assume for simplicity that the Lie algebroid removes the singularity of a presymplectic form, so that $\Pi_{\alpha\beta}$ is invertible on a dense subset of a neighborhood of $0$. It follows that on that subset $A_{i\alpha} = C_{i\alpha}\Pi_{\alpha\beta}$, where $C_{i\alpha} := A_{i\gamma} (\Pi^{-1})_{\gamma\alpha}$. Then
\begin{equation*}
 a_i := \tilde{a}_i - C_{i\beta} \tilde{b}_{\beta} \,,\quad
 b_\alpha := \tilde{b}_\alpha 
\end{equation*}
is a basis of local sections of the form of Eq.~\eqref{eq:StandardFrame}, where
\begin{equation*}
 \omega^{xy}_{i\alpha} := - C_{i\alpha} \,,\quad
 \omega^{xx}_{ij} := B_{ij} - C_{i\alpha} \Pi_{\alpha\beta} C_{j\beta} \,.
\end{equation*}
A priori, $C_{i\alpha}$ and, hence, $a_i$ is only defined on the regular part. But since $\{\frac{\partial}{\partial x^i } \}$ spans the tangent space of the leaf of $\rho(L)$ through $0$, the sections $a_i$ and, hence, $C_{i\alpha}$ extends smoothly to $M_\mathrm{sing}$. Let $\omega$ be defined in terms of $\omega^{xx}_{ij}$ and $\omega^{xy}_{i\alpha}$ as in Eq.~\eqref{eq:Defomega}. It can be checked by explicit calculation that $[a_i, b_\alpha ] = 0$ is equivalent to $d\omega = 0$.

\bibliographystyle{plain}
\bibliography{DiracStructures}

\end{document}